\documentclass{amsart}

\usepackage{amssymb}

\newcommand{\Q}{\mathbb{Q}}
\newcommand{\R}{\mathbb{R}}
\newcommand{\Z}{\mathbb{Z}}
\newcommand{\C}{\mathbb{C}}

\renewcommand{\H}{\mathbb{H}}

\newcommand{\SL}{\mathbf{SL}}
\newcommand{\GL}{\mathbf{GL}}

\renewcommand{\sl}{\mathfrak{sl}}

\newcommand{\supp}{\mathrm{supp}}

\newcommand{\vol}{\mathrm{vol}_\mu}

\newcommand{\tr}{\mathrm{tr}}

\newcommand{\Ga}{\Gamma}
\newcommand{\ga}{\gamma}
\newcommand{\D}{\mathcal{D}}
\newcommand{\U}{\mathcal{U}}

\def\peter[#1]{\langle #1\rangle}

\usepackage{color}

\usepackage{hyperref}

\usepackage{aliascnt}

\newtheorem{theorem}{Theorem}[section]

\newaliascnt{lemma}{theorem}
\newtheorem{lemma}[lemma]{Lemma}
\aliascntresetthe{lemma}

\newaliascnt{cor}{theorem}
\newtheorem{cor}[cor]{Corollary}
\aliascntresetthe{cor}

\newaliascnt{prop}{theorem}

\aliascntresetthe{prop}

\newaliascnt{con}{theorem}
\newtheorem{con}[con]{Conjecture}
\aliascntresetthe{con}

\theoremstyle{remark}
\newaliascnt{remark}{theorem}
\newtheorem{remark}[remark]{Remark}
\aliascntresetthe{remark}

\numberwithin{equation}{section}

\begin{document}

\title[Small generators of cocompact arithmetic Fuchsian groups]{Small generators of cocompact arithmetic Fuchsian groups}

\author{Michelle Chu}
\address{Department of Mathematics, University of Texas at Austin, Austin, TX 78750, USA}
\curraddr{}
\email{mchu@math.utexas.edu}
\thanks{M. Chu was supported in part by  NSF Grant DMS-1148490.}

\author{Han Li}
\address{Department of Mathematics and Computer Sciences, Wesleyan University, Middletown, CT 06457, USA}
\email{hli03@wesleyan.edu}
\thanks{H. Li was supported in part by an AMS Simons Travel Grant.}

\subjclass[2010]{20H10, 11F06}

\begin{abstract}
In the study of Fuchsian groups, it is a nontrivial problem to determine a set of generators. Using a dynamical approach we construct for any cocompact arithmetic Fuchsian group a fundamental region in $\SL_2(\R)$ from which we determine a set of small generators.
\end{abstract}

\maketitle

\section{Introduction}\label{intro}

Arithmetic Fuchsian groups form a subclass of Fuchsian groups with a special connection to number theory and the theory of automorphic forms. These groups are necesarily of finite covolume and thus finitely generated. However, as with general Fuchsian groups, it is a nontrivial problem to determine a set of generators. One way to do this involves constructing a polygonal fundamental domain in $\H^2$ and listing a set of generators corresponding to the side pairings of the polygon. Johansson, Voight, and Page produced algorithms to determine fundamental domains for groups of units in a maximal order of a quaternion algebra which are arithmetic Fuchsian (see \cite{J,V}) or Kleinian (see \cite[Chapter 3]{Page}). Macasieb determined fundamental domains for derived arithmetic Fuchsian groups of genus 2 (see \cite{M}). One could theoretically determine a set of generators for other groups in the commensurability class, but this would be impractical in general.

We also note that Chinburg and Stover \cite{CS} obtain bounds for small generators S-units of division algebras using lattice point methods. However, their generators are small in the division algebra and their representatives in $\GL_n$ do not neccesarily have small entries.

In this paper, we consider any cocompact arithmetic Fuchsian group and use dynamical techniques to construct a fundamental region in $\SL_2(\R)$ from which we determine a set of small generators. Our methods are reminiscent of methods used by Burger and Schroeder in \cite{BS}. In his thesis, Page uses a Riemannian geometry approach to give results of the same type in \cite[Chapter 2]{Page}.

\subsection{Statement of Results}
Let $G=\SL_2(\R)$ with a fixed Haar measure $\mu$, and $\Ga$ be a cocompact arithmetic Fuchsian group (see \autoref{arith}). Our main result gives a bound on norms of the generators in terms of the degree of the invariant trace field $k$, the volume of $G/\Ga$, and the spectral gap of the Laplace-Beltrami operator on $\H^2/\Ga$.

\begin{theorem}\label{main}
There exists a constant $C>0$ depending only on the Haar measure $\mu$ and satisfying the following property.
Let $\Ga$ be a cocompact arithmetic Fuchsian subgroup of $G$ with invariant trace field $k$. Then $\Ga$ is generated by the finite subset
\begin{equation}\label{eqmain}
\left\lbrace \ga\in\Ga : \|\ga \| 
< C \cdot ([k:\Q]^{\frac{60}{1-\sqrt{1-4\lambda}}}) \cdot \vol(G/\Ga)^{\frac{6}{1-\sqrt{1-4\lambda}}} \right\rbrace.
\end{equation}
Here $\vol(G/\Ga)$ is the covolume of $\Ga$ with respect to the measure $\mu$, and $\lambda=\min\{\frac{1}{4}, \lambda_1(\Ga)\}$ where $\lambda_1(\Ga)$ is the smallest non-zero eigenvalue of the Laplace-Beltrami operator on $\mathbb{H}^2/\Ga$.
\end{theorem}

Note that over a finite-dimensional space the choice of norm $\|\cdot\|$ does not matter. However, for completeness we take the $L^\infty$-norm. The coefficient $C$ in \autoref{main} can be explicitly calculated by following the methods of the proofs. Also, the result can be strengthened for certain families of arithmetic Fuchsian groups $\Ga$, as we now describe.

\begin{cor}\label{congruence}
If $\Ga$ is a cocompact congruence arithmetic Fuchsian group then it is generated by the finite subset
\begin{equation}
\left\lbrace \ga\in\Ga : \|\ga \| 
< C \cdot ([k:\Q]^{\frac{384}{5}}) \cdot \vol(G/\Ga)^{\frac{192}{25}} \right\rbrace.
\end{equation}
\end{cor}

\begin{cor}\label{tf}
If $\Ga$ is torsion free cocompact arithmetic Fuchsian group then it is generated by the finite subset
\begin{equation}
\left\lbrace \ga\in\Ga : \|\ga \| 
< C \cdot (\log([k:\Q])^{\frac{180}{1-\sqrt{1-4\lambda}}}) \cdot \vol(G/\Ga)^{\frac{6}{1-\sqrt{1-4\lambda}}} \right\rbrace.
\end{equation}
\end{cor}

\section{Acknowledgements}
The authors thank Alan W. Reid for helpful discussions and suggestions. They also thank Gregory A. Margulis for the techniques in his upcoming joint work with the second author. Finally, the authors would like to thank the referee for being incredibly thorough and for helpful comments and suggestions that improved the quality of this paper and the clarity of the arguments.

\section{Arithmetic Fuchsian groups}\label{arith}
We recall the definition and some properties of arithmetic Fuchsian groups. 

\subsection{Definition and properties}\label{arith} Let $k$ be a totally real number field and $A$ a quaternion algebra over $k$ ramified at all archimedean places except one. Let $R_k$ be the ring of integers in $k$, $\mathcal{O}$ an $R_k$-order in $A$, and $\mathcal{O}^1$ the elements of $A$-norm 1 in $\mathcal{O}$. Let $\rho$ be a $k$-embedding of $A$ in $M_2(\R)$, $2\times 2$ matrices over $\R$. Then $\rho(\mathcal{O}^1)$ is a Fuchsian group of finite covolume. By definition, a subgroup $\Ga$ in $\SL_2(\R)$ is an arithmetic Fuchsian group if it is commensurable, up to conjugation, with some such $\rho(\mathcal{O}^1)$.

An arithmetic Fuchsian group $\Ga$ is called a congruence subgroup if there is a maximal order $\mathcal{O}$ and an integral 2-sided $\mathcal{O}$-ideal $I$ of $\mathcal{O}$ in $A$ such that $\Ga$ contains $\rho(\mathcal{O}^1(I))$ where $\mathcal{O}^1(I)=\{a\in\mathcal{O}^1: a-1\in I \}$.

The field $k$ in the definition of an arithmetic Fuchsian group $\Ga$ is a commensurability invariant and is recovered from $\Ga$ by $k\Ga=\Q(\tr(\Ga^{(2)})$, where $\Ga^{(2)}$ is the subgroup generated by the squares of elements in $\Ga$. The number field $k\Ga$ is the invariant trace field of $\Ga$. The quaternion algebra is also recovered from $\Ga$ by $A\Ga=\{\sum a_i\ga_i: a_i\in k\Ga, \ga_i\in\Ga^{(2)} \}$. For the remainder of the paper, we assume $\Ga$ is cocompact. The group $\Ga$ is cocompact when either $k\neq\Q$ or when $k=\Q$ and $A$ is a division algebra. We note also that for any $\ga\in\Ga$, $\tr\ga$ is an algebraic integer of degree at most 2 over $k\Ga$. 

For more on arithmetic Fuchsian groups, including proofs of the statements mentioned, the reader is referred  to \cite[Chapter 8]{MR}.

\subsection{Lehmer's and Salem's conjectures}
The Mahler measure of an algebraic integer $\theta$ is defined by
\begin{equation}\label{mahlermeasure}
M(\theta)= |a_0| \prod \max (1,|\theta_i|)
\end{equation}
where $a_0$ is the leading coefficient of the minimal polynomial for $\theta$ over $\Z$ and $\theta_i$ ranges over the algebraic conjugates of $\theta$.

In 1933 Lehmer posed the following conjecture:

\begin{con}\cite{lehmer} There exists a universal lower bound $m>1$ such that $M(\theta)\geq m$ for all $\theta$ which are not roots of unity.
\end{con}

A Salem number is a real algebraic integer $\theta>1$ such that $\theta^{-1}$ is a conjugate and all other conjugates lie on the unit circle. For Salem numbers, Lehmer's conjecture restricts to the Salem conjecture.

\begin{con} There exists $m_S>1$ such that if $\theta$ is a Salem number then $M(\theta)\geq m_S$.
\end{con}

The Salem conjecture is equivalent to the short geodesic conjecture of arithmetic 2-manifolds, which states that there is a universal lower bound for translation length for hyperbolic elements in arithmetic Fuchsian groups (see \cite{NR}). Assuming the Salem conjecture, the bounds in \autoref{main} can be modified to depend only on the covolume and the spectral gap whenever the group $\Ga$ contains no elliptic elements. 

\begin{cor}\label{salem}
If the Salem conjecture is true and $\Ga$ is a torsion free, cocompact, arithmetic Fuchsian subgroup of $G$. Then $\Ga$ is generated by the finite subset
\begin{equation}
\left\lbrace \ga\in\Ga : \|\ga \| 
< C \cdot (\log(m_S))^{\frac{-60}{1-\sqrt{1-4\lambda}}} \cdot \vol(G/\Ga)^{\frac{6}{1-\sqrt{1-4\lambda}}} \right\rbrace.
\end{equation}
\end{cor}

\section{Preparation}

\subsection{Injectivity radius}

Throughout this section, let $d$ be the degree of the the invariant trace field of $\Ga$.

\begin{lemma}\label{trace} 
Let $\Ga$ be a cocompact arithmetic Fuchsian group with invariant trace field of degree $d$ and let
\begin{equation}\label{W_d}
W_d = \left\{ g\in G :  2 \cos\left( \frac{\pi}{2d} \right) <|\tr(g)|< 2 \cosh \left( \frac{1}{16} \left( \frac{\log(\log(2d))}{\log(2d)} \right)^3 \right) \right\}.
\end{equation}
Then for any $g\in G$, $g\Ga g^{-1}\cap W_d = \{1\}$
\end{lemma}

\begin{proof}
By the cocompactness condition, we consider only hyperbolic and elliptic elements in $\Ga$.

If $\ga$ is hyperbolic then also $\ga^2$ is hyperbolic, with $|\tr(\ga^2)|>2$ and $\tr(\ga^2)\in k\Ga$. Let $u_\ga$ be the root of the characteristic polynomial $x^2-\tr(\ga)x+1$ of $\ga$ which lies outside the unit circle (respectively, let $u_{\ga^2}$ be the root of the characteristic polynomial $x^2-\tr(\ga^2)x+1$ of $\ga^2$ outside the unit circle). We have $u_{\ga^2}$ real and exactly two of its algebraic conjugates lie off the unit circle, namely $u_{\ga^2}$ and $u_{\ga^2}^{-1}$ (see e.g. \cite[Lemma 12.3.1]{MR}). Therefore, $\log M(u_{\ga^2})=\log|u_{\ga^2}|=2\log|u_\ga|$ where $M(u_{\ga^2})$ is the Mahler measure (see \autoref{mahlermeasure}).
As a hyperbolic element, $\ga$ acts on the hyperbolic plane by translating along its invariant axis a distance of $2\log|u_\ga|$. It then follows that $|\tr(\ga)|=2 \cosh\left( \frac{1}{4} \log M(u_{\ga^2}) \right)$ (see e.g. \cite[Lemma 12.1.2]{MR}). Since $u_{\ga^2}$ lies in an extension of degree 2 over $k\Ga$, it has algebraic degree bounded above by $2d$. We now apply the main theorem of \cite{voutier} which states that any algebraic integer $\theta$ of degree at most $2d$ has \begin{equation}
\log M(\theta)>\frac{1}{4}\left( \frac{\log(\log(2d))}{\log(2d)} \right)^3
\end{equation}
to get the following bound for for the trace of a hyperbolic $\ga\in\Ga$
\begin{equation}
|\tr(g)|> 2 \cosh \left( \frac{1}{16} \left( \frac{\log(\log(2d))}{\log(2d)} \right)^3 \right).
\end{equation}

If $\ga$ is elliptic then by the discreteness of $\Ga$, $\ga$ has order $m$ and thus eigenvalues $\omega$ and $\omega^{-1}$ with $|\omega+\omega^{-1}|=|\tr(\ga)|<2$ and such that $\omega$ is a primitive $2m$-th root of unity. Since $\tr(\ga)$ lies in a field of degree at most 2 over $k\Ga$, $(\omega+\omega^{-1})$ has algebraic degree of at most $2d$, or equivalently, $\omega$ has algebraic degree of at most $4d$. Then $|\tr(\ga)|=|u_\ga+u_\ga^{-1}|$ is bounded above by $ |e^{\frac{2\pi i}{4d}}+e^{-\frac{2\pi i}{4d}}| = 2\cos(\frac{\pi}{2d}) $.

Take $W_d$ as in \autoref{W_d}. Since the trace is invariant under conjugation, for any $g\in G$ we have $g\Ga g^{-1}\cap W_d=\{1\}$.
\end{proof}

The injectivity radius of $G/\Ga$ at a point $x=h\Ga \in G/\Ga$ is related to the distance between the identity element and the nearest point in the lattice $h\Ga h^{-1}$. Since the trace of a matrix is invariant under conjugation, the bounds in \autoref{trace} for the traces of elements in any conjugate of $\Ga$ determine an injectivity radius for $G/\Ga$.

For any $\eta>0$ denote the ball of radius $\eta$ centered at the identity $1\in G$ by
$$ B_G(\eta) = \left\lbrace g\in g: \| g - 1 \| < \eta \right\rbrace . $$

\begin{lemma}
\label{radius}
There exists $\delta=\delta(d)=c d^{-2}>0$ with $c$ depending only on the Haar measure $\mu$, such that the map $B_G(\delta)\to G/\Ga$ given by $g\mapsto g.x$ is injective for any $x\in G/\Ga$.
\end{lemma}

\begin{proof}
Let
\begin{equation}
\delta_0 = \min \left\lbrace \cosh \left( \frac{1}{16} \left( \frac{\log(\log(2d))}{\log(2d)} \right)^3 \right)- 1 , 1 - \cos\left( \frac{\pi}{2d} \right)  \right\rbrace .
\end{equation}
Then for any $0<\delta_1<\delta_0$ and $g\in B_G(\delta_1)$,
\begin{equation*}
|g_{1,1}-1|< \delta_1 \hspace{10pt} \text{ and } \hspace{10pt} |g_{2,2}-1|< \delta_1 ,
\end{equation*}
where $g_{1,1}$ and $g_{2,2}$ are the $({1,1})$ and $({2,2})$ matrix entries of $g$.
Then $|\tr(g)-2| < 2\delta_1$ and $g\in W_d$.

By the Taylor series expansion of $1-\cos(\pi/2d)$,
\begin{equation}
 1 - \cos\left( \frac{\pi}{2d} \right) \approx \frac{\pi^2}{8d^2} -\frac{\pi^2}{384 d^4}.
\end{equation}
Since $1 - \cos\left( \frac{\pi}{2d} \right)$ goes to 0 at a faster rate than $\cosh \left( \frac{1}{16} \left( \frac{\log(\log(2d))}{\log(2d)} \right)^3 \right)- 1$ goes to 0, there exists a positive constant $c$ depending only on the measure $\mu$ such that $\delta(d):= c \cdot d^{-2} <\delta_0$ and such that for all $g_1,g_2\in B_G(\delta)$, we have $g_1^{-1}g_2\in B_G(\delta_0)$

For any $x=h\Ga\in G/\Ga$, if $g_1,g_2\in B_G(\delta)$ and $g_1.x=g_2.x$ then $g_1^{-1}g_2\in h\Ga h^{-1}$. Hence $g_1^{-1}g_2\in W_d$ and by \autoref{trace} it must be that $g_1=g_2$.
\end{proof}

\begin{remark}\label{torsionfree}
When $\Ga$ is torsion free, only the traces of hyperbolic elements are relevant. In this case the constant $\delta(d)$ in \autoref{radius} can be improved to $c{\left( \frac{\log(\log(2d))}{\log(2d)} \right)^6}$. For a cleaner statement, we use the estimate $\log(d)^{-4} \ll{\left( \frac{\log(\log(2d))}{\log(2d)} \right)^6}$.

If the Salem conjecture is true, there is a universal lower bound for the trace of a hyperbolic element in an arithmetic Fuchsian group. Then for $m_S$ the lower bound satisfying the Salem conjecture, the constant $\delta$ can be improved to $c \log^2(m_S)$ for all torsion free $\Ga$. In particular, this constant would be independent of $d$.
\end{remark}

\subsection{Spectral gap}
\label{spectral gap}
Let $\Ga$ be a cocompact Fuchsian group acting on the hyperbolic plane $\mathbb{H}^2$ from the right. Let $\lambda_1(\Ga)$ be the first non-zero eigenvalue of the Laplacian operator on the locally symmetric space $\mathbb{H}^2/\Ga$. Denote by $\lambda=\lambda(\Ga):=\min\{\frac{1}{4}, \lambda_1(\Ga)\}$. Let
\begin{equation}
\alpha_t=
\begin{pmatrix}
e^{\frac{t}{2}} & 0 \\
0 & e^{-\frac{t}{2}}
\end{pmatrix}
\in\SL_2(\R),
\qquad\qquad
\D=
\begin{pmatrix}
0 & 1 \\
-1 & 0
\end{pmatrix}\in\sl_2(\R).
\end{equation}

Denote by $\mu_{G/\Ga}$ the G-invariant probability measure supported on $G/\Ga$ for which, whenever $X\subset G/\Ga$ is a measurable set, we have $\vol(X)={\vol(G/\Ga)\cdot\mu_{G/\Ga}(X)}$.

\begin{lemma} \label{gap}
Let $\Ga$ be a cocompact lattice in $G$. Then for any smooth functions $\varphi$, $\psi$ in $L^2_0(G/\Ga)=\{f\in L^2(G/\Ga) : \int f d\mu_{G/\Ga} = 0 \}$ it holds
\begin{equation}
|\peter[\alpha_t.\varphi, \psi]| 
\ll (e^{-t/3})^{1-\sqrt{1-4\lambda}} \cdot \|\D.\varphi\| \cdot \|\D.\psi\| .
\end{equation}
\end{lemma}

\begin{proof}
This lemma follows from work of Ratner in {\cite{flow}} but is given explicitly in \cite[Corollary 2.1]{explicit} and with explicit constant. We use the estimate $t(e^{-t/2})\ll (e^{-t/3})$ for a cleaner statement.
\end{proof}

\subsection{Translates}
Given any two points in the $G/\Ga$, some element $h\in G$ will translate one to the other. We use the injectivity radius, the spectral gap, and the covolume to give an upper bound on the norm of some such element $h$.

\begin{theorem}\label{trans} 
For any $g_1\Ga, g_2\Ga \in G/\Ga$ there is some $h\in G$ and positive number $c$ such that $h.g_1\Ga=g_2\Ga$ and 
\begin{equation}
\|h\| < c \cdot \vol(G/\Ga)^{\frac{3}{1-\sqrt{1-4\lambda}}} \cdot \delta^{\frac{-15}{1-\sqrt{1-4\lambda}}},
\end{equation}
where $\delta$ is in \autoref{radius} and $\lambda$ is as in \autoref{spectral gap}. 
\end{theorem}

\begin{proof}
Let $\varphi$ and $\psi$ be nonnegative, smooth bump functions supported in $B_G(\delta).g_1\Ga$ and $B_G(\delta).g_2\Ga$ respectively, and satisfying
\begin{equation}
\mu_{G/\Ga}(\varphi)=\mu_{G/\Ga}(\psi)=\vol(G/\Ga)^{-1}, 
\qquad\qquad
\|\D.\varphi\| , \|\D.\psi\| \ll \vol(G/\Ga)^{-\frac{1}{2}} \cdot \delta^{-\frac{5}{2}}.
\end{equation}
These functions can be constructed as in \cite[Lemma 2.4.7]{KM}.
Applying \autoref{gap} to the functions $\varphi-\mu_{G/\Ga}(\varphi)$ and $\psi-\mu_{G/\Ga}(\psi)$, we get
\begin{equation}\label{ineq}
\left|\peter[\alpha_t.\varphi, \psi] - \mu_{G/\Ga}(\varphi) \cdot \mu_{G/\Ga}(\psi) \right| 
\ll (e^{-t/3})^{1-\sqrt{1-4\lambda}} \cdot \|\D.\varphi\| \cdot \|\D.\psi\| .
\end{equation}
Then for $t>0$ such that 
\begin{equation*}
\|\alpha_t\| \leq e^t \ll \vol(G/\Ga)^{\frac{3}{1-\sqrt{1-4\lambda}}} \cdot \delta^{\frac{-15/2}{1-\sqrt{1-4\lambda}}} , 
\end{equation*}
we must have $\peter[\alpha_t.\varphi, \psi]\neq 0$, since $\mu_{G/\Ga}\varphi \cdot \mu_{G/\Ga}\psi$ will dominate the error term in \autoref{ineq}.
Therefore, the set
\begin{equation*}
\left( \alpha_t. \supp(\varphi) \cap \supp(\psi) \right) \subset \left( \alpha_t. (B_G(\delta).g_1\Ga) \cap (B_G(\delta).g_2\Ga) \right)
\end{equation*}
is not empty. There is some $h\in B_G(\delta)^{-1}.\alpha_t. B_G(\delta)$ for which $h.g_1\Ga=g_2\Ga$.
\end{proof}

\section{Proof of Results and Concluding Remarks}

\begin{proof}[Proof of \autoref{main}]\label{pfmain}
Define the set
\begin{equation}
\U = \left\lbrace h\in G : \|h\| \leq \vol(G/\Ga)^{\frac{3}{1-\sqrt{1-4\lambda}}} \cdot \delta^{\frac{-15}{1-\sqrt{1-4\lambda}}} \right\rbrace .
\end{equation}
By \autoref{trans}, for any $g\in G$, there is some $h\in \U$ such that $h.\Ga=g.\Ga$. That is, we can write $g$ as $h\ga$ for some $\ga\in\Ga$. Therefore, we have constructed an open set $\U\subset G$ such that $\U\Ga=G$. It follows that $\{\ga\in \Ga: \U\cdot \ga\cap \U\neq \emptyset\}$ generates $\Ga$ (see e.g. \cite[Lemma 6.6]{BHC}). Then $\Ga$ is generated by the bigger set $\U^{-1}\U\cap\Ga$, whose elements have norm bounded above by $\vol(G/\Ga)^{\frac{6}{1-\sqrt{1-4\lambda}}} \cdot \delta^{\frac{-30}{1-\sqrt{1-4\lambda}}}$.

The proof of \autoref{main} then follows by replacing $\delta$ with $c[k:\Q]^{-2}$, as given in \autoref{radius} (where $d=[k:\Q]$).
\end{proof}

\begin{proof}[Proof of \autoref{congruence}]
The proof of \autoref{congruence} follows from the well-known result that 
\begin{equation}\label{spectralgap}
\lambda_1(\Gamma)\geq 975/4096
\end{equation}
for every congruence arithmetic Fuchsian group. For the sake of reader's convenience, we shall briefly explain (\autoref{spectralgap}) in the case when $k\Gamma=\Q$. 

By the Jacquet-Langlands correspondence (see \cite[Chapter 16]{automorphic}),
if $\Ga$ is a congruence arithmetic Fuchsian group with $k\Gamma=\Q$ then $\lambda_1(\Gamma)\geq\inf\{\lambda_1(\Lambda(N)):\ N\in\mathbb{N}\}$, where $\Lambda(N)=\{\gamma\in\SL_2(\Z): \gamma\equiv 1\mod N \}$. As of this writing we have the estimate, which is due to Kim-Sarnark \cite{kimsarnak},
\begin{equation}\label{ksbound}
\inf\{\lambda_1(\Lambda(N)):\ N\in\mathbb{N}\}\geq 975/4096.
\end{equation}

We note that, for the case when $k\Gamma$ is a number field, \autoref{spectralgap} follows from the Jacquet-Langlands correspondence together with the bounds on the Ramanujan conjecture over number fields (see \cite{blomerbrumley}).
\end{proof}

\begin{proof}[Proof of \autoref{tf} and \autoref{salem}]
The proof of \autoref{tf} follows from the proof of \autoref{main} by replacing $\delta$ with $c \log(d)^{-4}$ as in \autoref{torsionfree}. Furthermore, if we assume the Salem conjecture is true, then \autoref{salem} will follow by setting $\delta=c \log^2(m_S)$ as in \autoref{torsionfree}.
\end{proof}

\begin{remark}
The arguments in the proof of \autoref{main} can be modified to give a similar result for cocompact arithmetic Kleinian groups. All lemmas will go through except that one would need an analogue of \autoref{gap}.  An analogue of \autoref{gap} (or \cite[Corollary 2.1]{explicit}) does hold. However, more work would need to be done to get an explicit rate of the decay of correlations for functions on $\SL_2(\C)/\Gamma$, from the $\lambda_1$ of the corresponding hyperbolic 3-manifold $\mathrm{SU}(2)\backslash\SL_2(\C)/\Gamma$. At this time, we are not aware of such a statement in the current literature.

\end{remark}

\bibliographystyle{amsplain}

\bibliographystyle{plain}

\begin{thebibliography}{888}

\bibitem{blomerbrumley} V. Blomer, F. Brumley {\it On the Ramanujan conjecture over number fields.}
Ann. of Math. (2) 174 (2011), no. 1, 581-605. 

\bibitem{BHC} A. Borel, Harish-Chandra, {\it Arithmetic subgroups of algebraic groups.} 
Ann. of Math. (2) 75 (1962), 485-535.

\bibitem{BS} M. Burger, V. Schroeder, {\it Volume, diameter and the first eigenvalue of locally symmetric spaces of rank one. } J. Differential Geom. 26 (1987), no. 2, 273-284. 

\bibitem{CS} T. Chinburg, M. Stover, {\it Small generators for S-unit groups of division algebras.} New York J. Math. 20 (2014), 1175-1202.

\bibitem{automorphic} H. Jacquet, R. Langlands, {\it Automorphic forms on GL(2).} Lecture Notes in Mathematics, Vol. 114. Springer-Verlag, Berlin-New York, 1970.

\bibitem{J} S. Johansson, {\it On fundamental domains of arithmetic Fuchsian groups.} Math. Comp. 69 (2000), no. 229, 339-349.

\bibitem{spectralgap} D. Kelmer, L. Silberman, {\it A uniform spectral gap for congruence covers of a hyperbolic manifold.}
Amer. J. Math. 135 (2013), no. 4, 1067-1085.

\bibitem{kimsarnak} H. Kim, P. Sarnak, {\it Refined estimates towards the Ramanujan and Selberg conjectures, Appendix to Functoriality for the exterior square of GL4 and the symmetric fourth of GL2.}
J. Amer. Math. Soc. 16 (2003), no. 1, 139-183.

\bibitem{KM} D.Y. Kleinbock, G. A. Margulis, {\it Bounded orbits of nonquasiunipotent flows on homogeneous spaces.}
Amer. Math. Soc. Transl. (2) 171 (1996), 141-172.

\bibitem{lehmer} D.H. Lehmer, {\it Factorization of certain cyclotomic functions.}
Ann. Math. (2) 34 (1933), no. 3, 461-479.

\bibitem{M} M. Macasieb, {\it Derived arithmetic Fuchsian groups of genus two.} Experiment. Math. 17 (2008), no. 3, 347-369.

\bibitem{MR} C. Maclachlan, A. W. Reid, { The Arithmetic of Hyperbolic 3-Manifolds.}
{\it Spring-Verlag,} New York (2003).

\bibitem{explicit} C. Matheus, {\it Some quantitative versions of Ratner's mixing estimates.} Bull. Braz. Math. Soc. (N.S.) 44 (2013), no. 3, 469-488. 

\bibitem{NR} W. D. Neumann, A. W. Reid, {\it Arithmetic of hyperbolic manifolds.} {Topology '90 (Columbus, OH, 1990)}, 273-310, Ohio State Univ. Math. Res. Inst. Publ., 1, de Gruyter, Berlin (1992).

\bibitem{Page} A. Page, {\it M\'{e}thodes explicites pour les groupes arithm\'{e}tiques.} PhD dissertation, Universit\'{e} de Bordeaux (2014). 

\bibitem{flow} M. Ratner, {\it The rate of mixing for geodesic and horocycle flows.}
Ergodic Theory Dynam. Systems 7 (1987), no. 2, 267-288.

\bibitem{V} J. Voight, {\it Computing fundamental domains for Fuchsian groups.} J. Th\'{e}or. Nombres Bordeaux 21 (2009), no. 2, 469-491.

\bibitem{voutier} P. Voutier, {\it An effective lower bound for the height of algebraic numbers.}
Acta Arith. 74 (1996), no. 1, 81-95.

\end{thebibliography}

\end{document}